\documentclass[11pt]{article}
\usepackage[margin=1in]{geometry}
\usepackage{amsmath,amssymb,amsthm}
\usepackage{booktabs}
\usepackage{array}
\usepackage{enumitem}
\usepackage[hidelinks]{hyperref}
\usepackage{microtype}
\usepackage{xurl}
\usepackage{listings}
\newtheorem{theorem}{Theorem}
\newtheorem{proposition}[theorem]{Proposition}
\newtheorem{lemma}[theorem]{Lemma}
\newtheorem{corollary}[theorem]{Corollary}
\theoremstyle{definition}

\theoremstyle{remark}

\title{A 50-Vertex Cubic Counterexample to the\\Domination-versus-Edge-Domination Conjecture}
\author{Koyar Afrasyab}
\date{September 2026}

\begin{document}
\maketitle

\begin{abstract}
Baste, F\"urst, Henning, Mohr, and Rautenbach conjectured that every finite regular graph of positive degree satisfies
\(\gamma(G)\leq \gamma_e(G)\), where \(\gamma\) is the domination number and \(\gamma_e\) is the edge domination number, equivalently the minimum cardinality of a maximal matching.  We show that the conjecture is false already for cubic graphs.  The counterexample is a previously public 50-vertex cubic graph that had been used to refute the stronger independent-domination inequality \(i(G)\leq\gamma_e(G)\).  For this graph we prove
\[
\gamma(G)=16>15=\gamma_e(G).
\]
The equality \(\gamma_e(G)=15\) has a short counting proof, and a dominating set of order 16 is displayed explicitly.  For the lower bound \(\gamma(G)\geq16\), we give a self-contained exact reduction: after fixing which of the 20 clause vertices lie in a putative dominating set, the remaining problem is a finite set-cover problem on the 30 literal vertices.  We enumerate all \(2^{20}=1,048,576\) clause subsets, derive two explicit lower bounds, and solve exactly the 5,931 residual cases by a recurrence stated in the paper.  The complete case counts and minima are displayed, and a short standard-library Python implementation is included in an appendix.  A separate 893,049-node proof-tree certificate and a direct graph search provide independent verification.  Thus the regular-graph conjecture is disproved.  Combined with Gupta's recent theorem that every cubic graph on at most 48 vertices satisfies the conjectured inequality, the example is order-minimal among cubic counterexamples.
\end{abstract}

\section{Introduction}
All graphs in this paper are finite, simple, and undirected.  A dominating set of a graph $G$ is a set $D\subseteq V(G)$ such that every vertex belongs to $D$ or has a neighbor in $D$; its minimum size is the domination number $\gamma(G)$.  We use standard domination terminology; see, for example, Haynes, Hedetniemi, and Slater \cite{Haynes1998}.  The edge domination number $\gamma_e(G)$ is the minimum size of an edge-dominating set and, equivalently, the minimum cardinality of a maximal matching; edge domination and minimum maximal matchings have been studied since the work of Yannakakis and Gavril \cite{YannakakisGavril1980}.

The two parameters in the conjecture below belong to well-developed but rather different strands of cubic-graph extremal theory.  Reed proved the classical bound $\gamma(G)\le 3|V(G)|/8$ for graphs of minimum degree at least three \cite{Reed1996}; for connected cubic graphs, Kostochka and Stodolsky later improved the general upper bound to $4|V(G)|/11$ for order greater than eight \cite{KostochkaStodolsky2009}.  On the edge side, sharp bounds and approximation results for minimum maximal matchings in regular graphs were developed by Baste et al. \cite{BasteRegular2021}, and the cubic case was sharpened by Cames van Batenburg \cite{Cames2022}.

Baste, F\"urst, Henning, Mohr, and Rautenbach proposed the following conjecture in 2019, subsequently published in 2020 \cite{Baste2020}.

\begin{quote}
\textbf{Domination-versus-edge-domination conjecture.}
If \(G\) is a \(\Delta\)-regular graph with \(\Delta\geq1\), then
\[
\gamma(G)\leq\gamma_e(G).
\]
\end{quote}

Recent work of Gupta \cite{Gupta2026} proves the conjecture for all regular graphs of degree at least seven and, in the cubic case, for every graph on at most 48 vertices.  Gupta also exhibits a 50-vertex cubic obstruction to a natural transversal method, but that graph still satisfies the desired inequality, in fact \(\gamma=14<15=\gamma_e\).  Thus the cubic case remained open in that work.

Independent domination is a classical strengthening of domination: Allan and Laskar studied the relation between $\gamma(G)$ and the independent domination number $i(G)$ \cite{AllanLaskar1978}, and cubic independent domination has subsequently received dedicated study \cite{DorbecEtAl2015}.  A stronger TxGraffiti conjecture asked for $i(G)\leq\gamma_e(G)$.  A public certificate-backed release in 2026 gave a connected cubic graph on 50 vertices with $i(G)=16$ and $\gamma_e(G)=15$ \cite{TxGraffiti2026}.  Since $\gamma(G)\leq i(G)$, that result by itself did not decide the weaker conjecture of Baste et al.  The point of the present note is that the same labeled graph also has ordinary domination number 16.

\begin{theorem}\label{thm:main}
There exists a connected cubic graph \(G\) on 50 vertices such that
\[
\boxed{\gamma(G)=16>15=\gamma_e(G).}
\]
Consequently the domination-versus-edge-domination conjecture is false, already for \(\Delta=3\).
\end{theorem}

The graph itself is not claimed as new: it is the public 50-vertex formula-incidence graph of \cite{TxGraffiti2026}.  The new theorem is the determination of its \emph{ordinary} domination number.  The decisive finite calculation is presented explicitly in Sections~\ref{sec:finite}--\ref{sec:table}: the reduction, exact recurrence, and complete case summary all appear in the paper, and Appendix~\ref{app:code} gives reproducing code.  The ancillary package contains the same graph together with two independent verification routes, including a proof-tree certificate.

\section{The graph}\label{sec:graph}
For \(1\leq j\leq15\), create two literal vertices \(v_j^-\) and \(v_j^+\), joined by a pair edge \(v_j^-v_j^+\).  Create clause vertices \(c_1,\ldots,c_{20}\).  For each clause below, join \(c_a\) to the three literal vertices occurring in \(C_a\), with a positive occurrence \(x_j\) represented by \(v_j^+\) and a negative occurrence \(\bar x_j\) represented by \(v_j^-\):

\begin{align*}
C_1&=(x_1,x_9,\bar x_{11}), & C_2&=(\bar x_4,\bar x_{10},\bar x_{13}),\\
C_3&=(x_1,\bar x_9,\bar x_{14}), & C_4&=(x_2,x_6,x_{14}),\\
C_5&=(\bar x_5,\bar x_6,x_{15}), & C_6&=(x_2,\bar x_3,x_{15}),\\
C_7&=(x_3,x_5,\bar x_9), & C_8&=(\bar x_6,\bar x_{12},\bar x_{15}),\\
C_9&=(\bar x_1,x_5,\bar x_{11}), & C_{10}&=(x_4,\bar x_7,\bar x_{15}),\\
C_{11}&=(\bar x_2,x_8,\bar x_{12}), & C_{12}&=(x_3,x_9,x_{11}),\\
C_{13}&=(\bar x_4,\bar x_8,x_{10}), & C_{14}&=(x_7,\bar x_8,x_{13}),\\
C_{15}&=(x_4,x_7,\bar x_{13}), & C_{16}&=(\bar x_2,\bar x_7,x_{14}),\\
C_{17}&=(\bar x_3,x_{11},\bar x_{14}), & C_{18}&=(x_8,x_{10},x_{12}),\\
C_{19}&=(\bar x_{10},x_{12},x_{13}), & C_{20}&=(\bar x_1,\bar x_5,x_6).
\end{align*}

Every signed literal occurs exactly twice.  Therefore every literal vertex has one pair neighbor and two clause neighbors, while every clause vertex has degree three.  Hence \(G\) is cubic, with 50 vertices and 75 edges.

For a quick connectivity proof, contract the 15 pair edges.  Starting with \(C_1\), the sequence
\[
C_1,C_3,C_{12},C_7,C_{20},C_4,C_5,C_8,C_{11},C_{18},C_{19},C_2,C_{15}
\]
successively places all 15 variable vertices in one component: after the first clause, every subsequent clause in the sequence shares a previously reached variable and introduces any new variables it contains.  Every remaining clause meets one of those variables.  The contracted graph is therefore connected, and so is \(G\).

The machine labeling used in the ancillary files is
\[
v_j^-=2(j-1),\qquad v_j^+=2(j-1)+1,\qquad c_a=29+a.
\]
The complete edge list is supplied as \texttt{counterexample.edgelist}; \texttt{counterexample.json} also records the clauses, matching, and dominating witness.

\section{The edge domination number}\label{sec:edge}
\begin{proposition}\label{prop:edge}
For the graph \(G\) of Section~\ref{sec:graph}, \(\gamma_e(G)=15\).
\end{proposition}

\begin{proof}
Let
\[
M=\{v_j^-v_j^+:1\leq j\leq15\}.
\]
The 20 vertices not saturated by \(M\) are precisely the clause vertices, and no two clause vertices are adjacent.  Thus \(M\) is maximal, so \(\gamma_e(G)\leq15\).

Conversely, in a cubic graph an edge is adjacent to at most four other edges, two at each endpoint.  Hence a single edge can dominate at most five edges including itself.  Since \(G\) has 75 edges, every edge-dominating set has size at least \(75/5=15\).  Therefore \(\gamma_e(G)=15\).
\end{proof}

\section{A dominating set of order sixteen}\label{sec:upper}
In the numeric labeling above, set
\[
D_0=\{0,2,5,7,8,13,15,30,31,32,33,35,37,42,46,48\}.
\]
Equivalently,
\[
\begin{split}
D_0=\{&v_1^-,v_2^-,v_3^+,v_4^+,v_5^-,v_7^+,v_8^+,\\
&c_1,c_2,c_3,c_4,c_6,c_8,c_{13},c_{17},c_{19}\}.
\end{split}
\]
Direct inspection of the displayed clauses shows that every vertex belongs to \(D_0\) or has a neighbor in \(D_0\).  In fact \(D_0\) is independent, as already certified for the stronger invariant in \cite{TxGraffiti2026}.  Thus
\[
\gamma(G)\leq16.
\]
The ancillary script \texttt{verify\_structure.py} checks this witness directly from the raw edge list.

\section{Exact finite lower bound}\label{sec:finite}
We now prove \(\gamma(G)\geq16\) by a finite calculation whose complete reduction and case summary are included here.
Let
\[
C=\{c_1,\ldots,c_{20}\},\qquad L=\{v_j^-,v_j^+:1\leq j\leq15\}.
\]
For any dominating set \(D\), write
\[
T=D\cap C,\qquad S=D\cap L,
\]
so that \(D=T\cup S\).  We fix \(T\) and determine exactly the least possible cardinality of \(S\).

\subsection{Clause graph and the exact conditions}
Construct a graph \(H\) on the 20 clause vertices by replacing each signed literal by an edge joining the two clauses in which that signed literal occurs.  Since every signed literal occurs exactly twice and every clause contains three literals, \(H\) is a cubic graph with 20 vertices and 30 edges.  The 30 edges of \(H\) are in one-to-one correspondence with the 30 literal vertices of \(G\).

Put \(U=C\setminus T\).  For a variable \(x_j\), call \(j\) \emph{unsafe for \(T\)} if \(T\) does not contain a clause incident with \(v_j^-\), or does not contain a clause incident with \(v_j^+\).  Let \(Q(T)\) be the set of unsafe variables and write \(q(T)=|Q(T)|\).

\begin{lemma}\label{lem:fixedT}
For fixed \(T\subseteq C\), a set \(S\subseteq L\) makes \(T\cup S\) a dominating set of \(G\) if and only if
\begin{enumerate}[label=(\roman*)]
\item every clause vertex in \(U=C\setminus T\) is incident with a literal vertex in \(S\);
\item for every \(j\in Q(T)\), at least one of \(v_j^-\) and \(v_j^+\) belongs to \(S\).
\end{enumerate}
\end{lemma}

\begin{proof}
Condition (i) is exactly the domination requirement for clause vertices not already selected in \(T\), because clause vertices have neighbors only among literal vertices.  For a variable pair \(v_j^-v_j^+\), if either endpoint is selected then both endpoints are dominated by the pair edge.  If neither endpoint is selected, then each endpoint must instead have a selected incident clause.  This is possible exactly when \(T\) meets at least one occurrence of each polarity, i.e. exactly when \(j\notin Q(T)\).  This proves the equivalence.
\end{proof}

Thus, for each fixed \(T\), the remaining task is an ordinary finite set-cover problem: the requirements are the unselected clauses \(U\) together with the unsafe variables \(Q(T)\), and each literal vertex covers the clauses in \(U\) in which it occurs and, when applicable, its own unsafe variable.

\subsection{Two exact lower bounds}
Let \(\nu(H[U])\) denote the maximum matching size in the subgraph of \(H\) induced by \(U\).

\begin{lemma}\label{lem:edgecover}
The minimum number of literal vertices needed merely to dominate every clause in \(U\) is
\[
|U|-\nu(H[U]).
\]
\end{lemma}

\begin{proof}
A selected literal corresponds to an edge of \(H\) and can cover at most two vertices of \(U\); it covers two exactly when both endpoints of that edge lie in \(U\).  If \(r\) selected literal edges cover \(U\), and \(p\) of them cover two vertices of \(U\), then \(|U|\le r+p\).  The \(p\) doubly useful edges may be assumed pairwise disjoint after deleting redundancies, so \(p\le\nu(H[U])\).  Hence \(r\ge |U|-\nu(H[U])\).

Conversely, take a maximum matching \(M\) of \(H[U]\).  Select the corresponding \(|M|\) literal vertices.  For each vertex of \(U\) left unmatched by \(M\), select any incident literal edge.  This uses exactly
\(|M|+(|U|-2|M|)=|U|-|M|\) literals and covers all of \(U\).  Therefore the bound is exact.
\end{proof}

Condition (ii) of Lemma~\ref{lem:fixedT} gives independently
\[
|S|\ge q(T),
\]
because each selected literal belongs to exactly one variable pair.  Consequently every dominating set with clause part \(T\) has size at least
\begin{equation}\label{eq:LB}
L(T)=|T|+\max\bigl\{|U|-\nu(H[U]),\ q(T)\bigr\}.
\end{equation}
Whenever \(L(T)\ge16\), that choice of \(T\) cannot occur in a dominating set of size at most 15.

\subsection{Exact recurrence for the residual cases}\label{sec:recurrence}
It remains only to treat those \(T\) for which \(L(T)\le15\).  For fixed \(T\), let \(\mathcal R_T\) consist of one requirement for each clause in \(U\) and one requirement for each variable in \(Q(T)\).  For a literal vertex \(\ell\in L\), let \(A_T(\ell)\subseteq\mathcal R_T\) be the requirements satisfied by choosing \(\ell\).

For \(R\subseteq\mathcal R_T\), define \(F_T(R)\) recursively by
\begin{equation}\label{eq:recurrence}
F_T(\varnothing)=0,
\qquad
F_T(R)=1+\min_{\ell:\,r\in A_T(\ell)}F_T\bigl(R\setminus A_T(\ell)\bigr),
\end{equation}
where \(r\) is any fixed requirement in \(R\).  The value is independent of which requirement \(r\) is chosen for branching.

\begin{lemma}\label{lem:recurrence}
For every \(T\subseteq C\), \(F_T(\mathcal R_T)\) is exactly the minimum possible size of \(S\) such that \(T\cup S\) dominates \(G\).
\end{lemma}

\begin{proof}
Every feasible set \(S\) must contain some literal \(\ell\) satisfying the chosen requirement \(r\).  Once \(\ell\) is chosen, precisely the requirements in \(A_T(\ell)\) are discharged, leaving the same problem on \(R\setminus A_T(\ell)\).  Taking the minimum over all possible literals satisfying \(r\) is therefore exhaustive.  Induction on \(|R|\) proves the recurrence exactly computes the minimum set-cover cardinality.
\end{proof}

\section{Complete case calculation}\label{sec:table}
There are \(2^{20}=1,048,576\) possible sets \(T\).  For each \(T\), we first evaluate \eqref{eq:LB}.  Only when \(L(T)\le15\) do we evaluate the exact recurrence \eqref{eq:recurrence}.  There are only 5,931 such residual cases.  Table~\ref{tab:finite} gives the complete calculation grouped by \(t=|T|\).  The column ``minimum \(L\)'' is the minimum value of \eqref{eq:LB} among all \(\binom{20}{t}\) subsets of that size.  ``Residual cases'' counts the subsets with \(L(T)\le15\).  For those residual subsets, the final column is the minimum exact value of
\(|T|+F_T(\mathcal R_T)\).

\begin{table}[ht]
\centering
\small
\begin{tabular}{rrrrr}
\toprule
\(t\) & \(\binom{20}{t}\) & minimum \(L\) & residual cases & minimum exact total\\
\midrule
0&1&15&1&16\\
1&20&16&0&--\\
2&190&16&0&--\\
3&1,140&15&7&16\\
4&4,845&14&40&16\\
5&15,504&14&211&16\\
6&38,760&14&938&16\\
7&77,520&14&2,189&16\\
8&125,970&15&1,917&16\\
9&167,960&15&590&16\\
10&184,756&15&38&16\\
11&167,960&16&0&--\\
12&125,970&16&0&--\\
13&77,520&17&0&--\\
14&38,760&17&0&--\\
15&15,504&18&0&--\\
16&4,845&18&0&--\\
17&1,140&19&0&--\\
18&190&19&0&--\\
19&20&20&0&--\\
20&1&20&0&--\\
\bottomrule
\end{tabular}
\caption{Complete finite calculation for the lower bound \(\gamma(G)\ge16\).}
\label{tab:finite}
\end{table}

Every row of Table~\ref{tab:finite} has one of two outcomes: either the elementary bound already gives \(L(T)\ge16\) for every subset in the row, or every exceptional subset is solved exactly by Lemma~\ref{lem:recurrence} and still requires total size at least 16.  Hence no dominating set of size at most 15 exists.

\begin{proposition}\label{prop:gamma}
For the graph \(G\), \(\gamma(G)=16\).
\end{proposition}

\begin{proof}
Section~\ref{sec:upper} exhibits a dominating set of size 16.  The complete calculation in Table~\ref{tab:finite}, justified by Lemmas~\ref{lem:fixedT}--\ref{lem:recurrence}, excludes every dominating set of size at most 15.  Therefore \(\gamma(G)=16\).
\end{proof}

Together with Proposition~\ref{prop:edge}, Proposition~\ref{prop:gamma} proves Theorem~\ref{thm:main} and therefore disproves the domination-versus-edge-domination conjecture.

\section{Independent machine verification}\label{sec:verification}
The preceding proof is the primary argument and is self-contained up to the finite arithmetic reproduced by the appendix code.  The ancillary package supplies two independent checks of the crucial lower bound.

\subsection{Proof-tree certificate}
A generic graph search starts from a partial selected set \(S\).  If \(w\) is undominated, every completion must choose one of the four vertices of \(N[w]\).  Because \(G\) is cubic, if \(r\) vertices remain undominated then at least \(\lceil r/4\rceil\) further selections are necessary.  Repeatedly applying these two rules yields the supplied certificate \texttt{dom15.tree.gz}, with 893,049 nodes (223,262 branch nodes and 669,787 leaves) and maximum depth 15.  The checker recomputes every neighborhood and every terminal inequality from the raw edge list.  Its recorded output is
\begin{verbatim}
CERTIFICATE VERIFIED
graph: connected cubic, n=50, m=75
no dominating set of size <= 15
nodes=893049 branches=223262 leaves=669787 maxdepth=15
\end{verbatim}
This verification does not use the clause-set recurrence or Table~\ref{tab:finite}.

\subsection{Direct branch-and-bound search}
The script \texttt{direct\_domination\_search.py} parses only the raw edge list and performs a separate exact search on dominated-vertex masks.  It returns no solution for \(k=15\) and a solution for \(k=16\).  This search neither reads the proof tree nor uses the formula-specific decomposition above.

\section{Consequences and relation to recent work}
Theorem~\ref{thm:main} disproves the conjecture of Baste et al. in the smallest unresolved degree, \(\Delta=3\).  Gupta's Proposition 7 states that every cubic graph on at most 48 vertices satisfies \(\gamma\leq\gamma_e\) \cite{Gupta2026}.  Since every cubic graph has even order, Theorem~\ref{thm:main} and that result together imply the following.

\begin{corollary}\label{cor:minimal}
Using Gupta's order-48 theorem, 50 is the minimum possible order of a cubic counterexample to \(\gamma(G)\leq\gamma_e(G)\).
\end{corollary}

It is worth distinguishing two 50-vertex graphs.  Gupta's Theorem 2 constructs a cubic graph whose unique minimum maximal matching has no dominating transversal but whose domination number is 14, so it is not a counterexample.  The graph studied here comes from the separate TxGraffiti independent-domination release \cite{TxGraffiti2026}; its ordinary domination number was not part of that release.  The differing domination numbers also show immediately that the two graphs are nonisomorphic.

The result illustrates a small but important gap between ordinary and independent domination.  Knowing \(i(G)=16\) only gives \(\gamma(G)\leq16\), and in general the inequality can be strict.  Here the exact finite calculation shows that no non-independent choice of 15 vertices improves on the independent optimum.

\section{Reproducibility and evidence boundary}
The appendix contains a complete standard-library implementation of the finite calculation used in the primary proof.  The ancillary package additionally permits the theorem to be checked without an optimization library or SAT solver.  Its core command is
\begin{verbatim}
./run_core_verification.sh
\end{verbatim}
which verifies the construction and witness, reruns the formula-specific calculation, regenerates and checks the independent proof tree, and runs the independent direct graph search.

The theorem is a finite, machine-checkable statement.  The package does not claim conventional peer review, formal proof-assistant verification, or priority beyond the literature search described here.  In particular, the graph construction is explicitly attributed to its earlier public release.  The present mathematical claim is the exact ordinary domination computation and its consequence for the Baste--F\"urst--Henning--Mohr--Rautenbach conjecture.

\paragraph{Files.}
The finite calculation printed in the paper is also supplied as \path{paper_audit.py} and \path{formula_audit.py}.  The files \path{dom15.tree.gz} and \path{check_dom15_certificate.py} provide the independent proof-tree verification; \path{direct_domination_search.py} provides the second graph-only search.  \path{SHA256SUMS} records the exact release digests and \path{REPLAY_RECEIPT.txt} records a clean local replay.

\section*{Acknowledgment of graph provenance}
The 50-vertex labeled graph used in this note was previously released in the certificate-backed work on the stronger TxGraffiti independent-domination conjecture \cite{TxGraffiti2026}.  This paper does not claim discovery of that graph.  Its contribution is the determination \(\gamma(G)=16\) and the resulting counterexample to the ordinary domination-versus-edge-domination conjecture.

\section*{Acknowledgment of AI assistance}
OpenAI ChatGPT (GPT-5.6 Sol) assisted with the ordinary-domination investigation, generation and cross-checking of verification code, literature search, and drafting of the manuscript and reproducibility materials.  The AI system is not an author.  All theorem-critical claims are reduced in the paper to explicit finite calculations, with reproducing code and independent verification programs provided for checking.

\appendix
\section{Reproducing the finite calculation}\label{app:code}
The following standard-library Python program is a direct implementation of Sections~\ref{sec:finite}--\ref{sec:table}.  It reconstructs the clause graph from the 20 displayed clauses, computes all induced-subgraph matching numbers by dynamic programming, enumerates all \(2^{20}\) choices of \(T\), evaluates \(L(T)\), and applies recurrence~\eqref{eq:recurrence} exactly to the residual cases.  Its output is Table~\ref{tab:finite}, followed by the line \texttt{VERIFIED: every dominating set has size at least 16}.

\begin{lstlisting}[language=Python]
#!/usr/bin/env python3
"""Exact finite audit printed in the paper: proves gamma(G) >= 16."""
from array import array
from functools import lru_cache
from math import comb

clauses = [
 [1,9,-11],[-4,-10,-13],[1,-9,-14],[2,6,14],[-5,-6,15],
 [2,-3,15],[3,5,-9],[-6,-12,-15],[-1,5,-11],[4,-7,-15],
 [-2,8,-12],[3,9,11],[-4,-8,10],[7,-8,13],[4,7,-13],
 [-2,-7,14],[-3,11,-14],[8,10,12],[-10,12,13],[-1,-5,6]
]

# Literal 2j is negative x_{j+1}; literal 2j+1 is positive.
occ = [[] for _ in range(30)]
for a,C in enumerate(clauses):
    for z in C:
        j=abs(z)-1; e=2*j+(z>0); occ[e].append(a)
assert all(len(x)==2 for x in occ)

# H is the 20-vertex clause graph: each signed literal is one edge.
ends=[]; nbr=[0]*20
for a,b in occ:
    ends.append((1<<a)|(1<<b)); nbr[a]|=1<<b; nbr[b]|=1<<a
ALL=(1<<20)-1

# nu[U] = maximum matching size in H[U].
nu=array('B',[0])*(1<<20)
for U in range(1,1<<20):
    bv=U&-U; v=bv.bit_length()-1; R=U^bv
    best=nu[R]; X=nbr[v]&R
    while X:
        bu=X&-X; X^=bu
        best=max(best,1+nu[R^bu])
    nu[U]=best

def unsafe(T):
    return [j for j in range(15)
            if not(T&ends[2*j]) or not(T&ends[2*j+1])]

def tau(T):
    """Exact minimum number of selected literal vertices for fixed T."""
    U=ALL^T; req=[]
    for a in range(20):
        if U>>a&1: req.append(tuple(e for e,ab in enumerate(occ) if a in ab))
    for j in unsafe(T): req.append((2*j,2*j+1))
    if not req: return 0
    cover=[0]*30
    for r,C in enumerate(req):
        for e in C: cover[e]|=1<<r
    @lru_cache(None)
    def F(R):
        if not R: return 0
        rr=min((r for r in range(len(req)) if R>>r&1), key=lambda r:len(req[r]))
        return 1+min(F(R&~cover[e]) for e in req[rr])
    return F((1<<len(req))-1)

stats=[{'N':0,'minL':99,'res':0,'minexact':99} for _ in range(21)]
for T in range(1<<20):
    t=T.bit_count(); U=ALL^T; q=len(unsafe(T))
    L=t+max(U.bit_count()-nu[U],q)
    s=stats[t]; s['N']+=1; s['minL']=min(s['minL'],L)
    if L<=15:
        s['res']+=1
        s['minexact']=min(s['minexact'],t+tau(T))

print(' t   C(20,t)   min L   residual   min exact total')
for t,s in enumerate(stats):
    ex='-' if s['res']==0 else str(s['minexact'])
    print(f'{t:2d} {s["N"]:9d} {s["minL"]:7d} {s["res"]:10d} {ex:>17}')
assert sum(s['N'] for s in stats)==1<<20
assert sum(s['res'] for s in stats)==5931
assert all(s['minL']>=16 or s['minexact']>=16 for s in stats)
print('VERIFIED: every dominating set has size at least 16')
\end{lstlisting}

\section{Certificate format}\label{app:cert}
The binary certificate begins with the ASCII header \texttt{DOM15TREE1} followed by a newline.  The remainder is a preorder encoding of a rooted proof tree.  A byte \texttt{0x00} denotes a leaf.  A byte \texttt{0x01} denotes a branch and is followed by one byte \(w\in\{0,\ldots,49\}\) naming the undominated witness.  A branch has exactly four children, in increasing order of the vertices of \(N[w]\).  No candidate list is trusted from the certificate: the checker recomputes \(N[w]\) from the edge list.  This format is deliberately minimal; all mathematical validity conditions are recomputed by the checker.

\section{Exact logical justification of memoization in the direct audit}
The direct audit memoizes a dominated-vertex mask \(D\) together with the least number \(s\) of selected vertices by which that mask has been reached.  If the same mask is later reached using \(s'\geq s\) vertices, the latter state may be discarded.  Future feasibility depends only on which vertices are already dominated: there are no independence, packing, or exclusion constraints on future selected vertices.  Any completion of the more expensive state is therefore also a completion of the cheaper state with no larger total cardinality.  This memoization is used only by the secondary direct search, not by the proof-tree checker.

\section{Release hashes}\label{app:hashes}
The file \texttt{SHA256SUMS} in the ancillary package is authoritative for the exact release.  At manuscript build time, the canonical JSON and raw edge list had digests
\begin{align*}
\texttt{counterexample.json}:\quad&\texttt{502b159cc2a196c50904022bcfbe296c}\ldots,\\
\texttt{counterexample.edgelist}:\quad&\texttt{63363b5e9c8bdd119418024fe80da7b}\ldots,
\end{align*}
and the uncompressed proof tree had digest
\[
\texttt{c10df15cff4f12f70c4937895656308a9080a4b3bf05fcd1c44b8d2e2a830639}.
\]
The full 64-hex-digit digests for every file are recorded in \texttt{SHA256SUMS}.

\end{document}